\documentclass[11pt, reqno]{amsart}

\usepackage{fourier}

\usepackage{amsthm}
\usepackage{amstext}
\usepackage{amssymb}
\usepackage{amsmath}

\usepackage{url}

\def\bl{\begin{lemma}}
\def\el{\end{lemma}}
\def\bth{\begin{theorem}}
\def\eth{\end{theorem}}
\def\bc{\begin{corollary}}
\def\ec{\end{corollary}}
\def\bcj{\begin{conjecture}}
\def\ecj{\end{conjecture}}
\def\bpr{\begin{proposition}}
\def\epr{\end{proposition}}
\def\bde{\begin{definition}}
\def\ede{\end{definition}}
\def\E{\mathbb{E}}

\newcommand{\be}{\begin{eqnarray}}
\newcommand{\ee}{\end{eqnarray}}
\newcommand{\eps}{{\mbox{$\epsilon$}}}

\newcommand{\Z}{{\mathbb Z}}

\renewcommand{\and}{\hbox{ {\rm and} }}

\newcommand{\C}{{\mathcal{C}}}
\newcommand{\prob}{\mbox{\bf P}}

\newcommand{\psrw}{\mbox{\bf p}_{\mbox{\rm srw}}}
\newcommand{\psaw}{\mbox{\bf p}_{\mbox{\rm saw}}}
\newcommand{\saw}{\hbox{{\rm SAW}}}
\newcommand{\pnbw}{\mbox{\bf p}_{\mbox{\rm nbw}}}

\newcommand{\lr}{\leftrightarrow}

\def\arrowfillCS#1#2#3#4{%
   \thickmuskip0mu\medmuskip\thickmuskip\thinmuskip\thickmuskip
   \relax#4#1\mkern-7mu%
   \cleaders\hbox{$#4\mkern-2mu#2\mkern-2mu$}\hfill
   \mkern-7mu#3
}
\def\lrfill{\arrowfillCS\leftarrow\relbar\rightarrow\relax}

\newtheorem{theorem}{Theorem}
\newtheorem{definition}{Definition}[section]
\newtheorem{lemma}{Lemma}[section]

\newtheorem{corollary}[lemma]{Corollary}
\newtheorem{proposition}[theorem]{Proposition}
\newtheorem{conjecture}[theorem]{Conjecture}

\theoremstyle{definition}
\numberwithin{equation}{section}

\begin{document}
\title{Non-amenable Cayley graphs of high girth have $p_c<p_u$ and mean-field exponents}
\author{Asaf Nachmias} \author{Yuval Peres}

\begin{abstract} In this note we show that percolation on non-amenable Cayley graphs of high girth has a phase of non-uniqueness, i.e., $p_c< p_u$. Furthermore, we show that percolation and self-avoiding walk on such graphs have mean-field critical exponents. In particular, the self-avoiding walk has positive speed.
\end{abstract}

\maketitle
\section{{\bf  Introduction}} One of the most well known conjectures in percolation theory, due to Benjamini and Schramm \cite{BS}, is that $p_c < p_u$ on any non-amenable Cayley graph. In other words, that any non-amenable Cayley graph exhibits a phase of percolation in which infinitely many infinite components exist with positive probability. In this note we show this holds under the additional assumption of high girth.

\begin{theorem} \label{percthm} For any $\rho < 1$ there exists $L>0$ such that if $G$ is a transitive graph with spectral radius at most $\rho$ and girth at least $L$, then
$$ p_c(G) < p_u(G) \, .$$
\end{theorem}

Our technique allows us to study the self-avoiding walk in the same setting. We remark that it is somewhat surprising that the analysis of this model relies on our percolation inequality Theorem \ref{bnp}. Recall that the self-avoiding walk of length $n$ on a graph $G$ is the uniform measure on simple paths (no vertex is visited more than once) of length $n$ starting at the origin. It is one of the easiest models to describe in statistical physics, yet is notoriously difficult to analyze or sample due to the lack of Markovian structure (see \cite{MS, BDcGS} for further details). We write $\saw(n)$ for the endpoint of the walk.

\begin{theorem} \label{sawthm}  For any $\rho < 1$ there exists $L>0$ such that if $G$ is a transitive graph with spectral radius at most $\rho$ and girth at least $L$, then there exists a constant $c>0$ such that for any vertex $x$
$$ \prob(\saw(n)=x)) \leq e^{-cn} \, .$$
Consequently, the self-avoiding walk has positive speed, that is, there exists some $c>0$ such that
$$ \prob\big ( d_{G}(0,\saw(n)) \leq c n \big ) \to 0 \, ,$$
where $d_G(x,y)$ denotes the graph distance in $G$ between $x$ and $y$.
\end{theorem}

Our results are similar in spirit to those of Schonmann \cite{Sc1} with one significant difference: Schonmann's results require that the spectral radius be small, while here we require the girth to be large. This allows us to apply the results for graphs in which the ratio of the Cheeger constant and the degree may be smaller than $2^{-1/2}$.
For example, Olshanskii and Sapir \cite{OS} and Akhmedov \cite{A} constructed Cayley graphs $G_n$ with girth going to $\infty$ and Cheeger constant uniformly bounded away from $0$, but this uniform bound may be arbitrarily close to $0$.

\subsection{Background} Given a graph $G$, two vertices $x,y$ of $G$ and an integer $n\geq 0$ we write $\psrw^n(x,y)$ for the probability that the simple random walk starting at $x$ visits $y$ at time $n$. Recall that the {\em spectral radius} $\rho\in[0,1]$ of a graph $G$ is defined by $\rho = \lim_{n \to \infty} (\psrw^{2n}(0,0))^{1/2n}$ (this limit always exists, see \cite{Woess}) and that the girth of $G$ is the length of the shortest cycle. A graph is said to be {\em non-amenable} if $\rho < 1$.

Given an infinite connected graph $G$ and $p\in[0,1]$ we define $p$-bond percolation to be the probability measure $\prob_p$ on subgraphs of $G$ obtained by independently deleting each edge with probability $1-p$ and retaining it otherwise. We call the retained edges {\em open} and the deleted edges {\em closed}. We say that two vertices $x$ and $y$ are connected if there exists a path of open edges in $G$ connecting $x$ and $y$ and denote this event by $x \lr y$. The connected component of $x$, denoted by $\C(x)$, is the set $\{y: x \lr y\}$. We define the critical percolation probability $p_c$ by
$$ p_c = \inf \big \{ p \in [0,1]\, : \, \prob_p (\exists \hbox{ {\rm an infinite connected component}} ) > 0 \big \} \, ,$$
and the uniqueness critical probability $p_u$ by
$$ p_u = \inf \big \{ p \in [0,1]\, : \, \prob_p (\exists \hbox{ {\rm  a unique infinite connected component}}) > 0 \big \} \, .$$

A beautiful argument due to Burton and Keane \cite{BK} shows that $p_c=p_u$ on any amenable transitive graph $G$. Benjamini and Schramm \cite{BS} conjectured that $p_c < p_u$ on any non-amenable Cayley graph. Pak and Smirnova-Nagnibeda \cite{PS} showed that for any non-amenable finitely generated group there exists a set of generators for which the resulting Cayley graph has $p_c < p_u$. Schonmann \cite{Sc1} showed that $p_c<p_u$ for Cayley graphs in which the ratio between the Cheeger constant and the degree is at least $2^{-1/2}$ and also for non-amenable Cayley graphs with more than one end (there $p_u=1$). Benjamini and Schramm \cite{BS2} showed that $p_c<p_u$ for transitive non-amenable planar graphs. We refer the reader to \cite{HJ} for further details.

As for the self-avoiding walk, Madras and Wu \cite{MW} showed that on some regular tilings of the hyperbolic plane the self-avoiding walk has positive speed. Duminil-Copin and Hammond \cite{DcH} show that the speed is zero on $\Z^d$ for any $d\geq 2$ and Madras \cite{M} gave a lower bound on the expected displacement of the self-avoiding walk on $\Z^d$. We expect that the statement of Theorem \ref{sawthm} holds for any non-amenable Cayley graph (see Question 5 of \cite{DcH}).



\subsection{Critical exponents} In addition, we show that percolation and the self-avoiding walk attain mean-field critical exponents on non-amenable graph of high girth. These exponents describe the behavior of the system at and near the critical point. Let us define the percolation critical exponents $\beta, \gamma$ and $\delta$, bearing in mind that in general there is no proof that they exist. See \cite{G} for further information.
\begin{align*}
\prob_p ( |\C(0)|=\infty ) & \asymp (p-p_c)^\beta  \, , & & p > p_c \\
\E _p |\C(0)| & \asymp (p-p_c)^{-\gamma}  \, , & & p < p_c \\
\prob_{p_c} ( |\C(0)| \geq n ) & \asymp n^{-1/\delta} \, ,
\end{align*}
where the symbol $\asymp$ implies that the ratio of both sides is bounded above and below away from $\infty$ and $0$. We say that a transitive graph $G$ satisfies the {\em triangle condition} at $p$  (which is usually $p_c$) if
$$ \sum_{x,y} \prob_p(0 \lr x) \prob_p(x\lr y) \prob_p(y \lr 0) < \infty \, .$$
Results in this area are usually of two types: proving that the triangle condition holds at $p_c$, and showing that graphs satisfying the condition have ``mean-field'' exponents, in particular $\beta=\gamma=1$ and $\delta=2$ which is the case for regular trees. Given a locally finite graph $G$, let $\Gamma$ be its group of automorphisms and denote by $S(x) = \{ \gamma \in \Gamma : \gamma x = x\}$ the stabilizer of $x$. We say a graph {\em unimodular} if for any pair of vertices $x,y$ we have $|S(x)y| = |S(y)x|$, see Chapter 8 of \cite{LP} for further details. In particular, any Cayley graph is unimodular. In the combined works of Aizenman, Barsky and Newman \cite{AB,AN,BA} it is shown that the triangle condition implies the graph has mean-field exponents when $G$ is a unimodular transitive graph (they proved it for $\Z^d$, but the proof works in the generality of unimodular transitive graphs, see the discussion around (3.14) in \cite{Sc1}). Here we show that the triangle condition holds for non-amenable graphs of high girth.

\begin{theorem} \label{percolationexponents} For any $\rho < 1$ there exists $L>0$ such that if $G$ is a regular graph with spectral radius at most $\rho$ and girth at least $L$, then the percolation triangle condition on $G$ holds at $p_c$. Hence, if $G$ is a transitive unimodular graph, then the critical exponents $\beta, \gamma, \delta$ exist with $\beta=\gamma=1$ and $\delta=2$.
\end{theorem}

Write $c_n$ for the number of self-avoiding paths of length $n$ starting at at the origin. Recall that the sequence $c_n$ is submultiplicative (see \cite{BDcGS}) hence the limit $\lim _{n \to \infty} n^{-1} \log c_n$ exists and equals $\inf_n n^{-1} \log c_n$. This number is commonly denoted by $\mu$. We also write $c_n(x)$ for the number of self-avoiding paths of length $n$ starting at the origin and ending at $x$, so $\psaw(0,x) = c_n(x)/c_n$ is the law of the location of the self-avoiding walk after $n$ steps. Write SAW$(n)$ for a random vertex distributed according to this law. The critical exponents $\gamma$ and $\nu$ associated with the self-avoiding walk are defined (as before, only when they exist) by:
\begin{align*}
\gamma = \lim_{n\to \infty} {\log c_n \mu^{-n} \over \log n} + 1 \, , \qquad \E [d_G(0, \hbox{{\rm SAW}}(n))] \asymp n^\nu \, .
\end{align*}

For $z \in [0,\mu^{-1})$ we define the sums
$$ G_z(x) = \sum_{n \geq 0} c_n(x) z^n \quad \and \quad \chi(z) = \sum_n c_n z^n = \sum_x G_z(x) \, .$$
Since $\lim c_n^{1/n} = \mu$ it is clear that both series converge and that $\mu^{-1}$ is the radius of convergence for $\chi(z)$. We say that a graph $G$ satisfies the self-avoiding walk {\em bubble condition} if
$$\lim_{z \to \mu^{-1}} \sum_{x \in G} G_{z}^2(x) < \infty \, .$$
The bubble condition for the self-avoiding walk is the analogue of the triangle condition. It was proven to hold for the integer lattice $\Z^d$ when $d\geq 5$ using the lace expansion by the seminal works of Brydges and Spencer \cite{BS} and Hara and Slade \cite{HaS92}. It is a useful condition since for any transitive graph it implies that
$$ \chi(z) \asymp {1 \over z - \mu^{-1}} \, ,$$
(see section 4.2 of \cite{BDcGS} or \cite{MS} for this implication --- there the proofs are for $\Z^d$ but a closer inspection shows that they only use transitivity). A standard Tauberian theorem (e.g., Lemma 6.3.3 in \cite{MS}) now implies that $\gamma=1$. 

Theorem \ref{sawthm} shows that $\nu=1$ in the setting of non-amenable graphs with high girth. Here we additionally show that the bubble condition holds as well.
\begin{theorem} \label{sawexponents} For any $\rho < 1$ there exists $L>0$ such that if $G$ is a transitive graph with spectral radius at most $\rho$ and girth at least $L$, then the self-avoiding walk bubble condition holds, whence $\gamma=1$.
\end{theorem}

\noindent {\em Remark.} We were not able to establish that $c_n=O( \mu ^n)$ in this setting. An estimate like that is known in $\Z^d$ but requires much more precise asymptotics on $\chi(z)$ as $z \to \mu^{-1}$ which are unavailable to us.

\section{Proofs}

The starting point of our proofs is the main result of \cite{BNP}
\begin{theorem}[Theorem 1 of \cite{BNP}]\label{bnp}  There exists a universal constant $C>0$ such that if $G$ is a non-amenable regular graph with degree $d$, girth $g$ and spectral radius $\rho < 1$, then
$$ p_c(G) \leq {1 \over d-1} + {C \log (1+ (1-\rho)^{-2}) \over dg} \, .$$
\end{theorem}
In particular, the statement of the theorem above implies that for any $\rho<1$ there exists $L>0$ such that if $G$ is a regular graph with spectral radius at most $\rho$ and girth at least $L$ we have
\be \label{perccond} p_c (d-1) \rho < 1 \, . \ee
In fact, we will prove the assertions of Theorems \ref{percthm}, \ref{sawthm}, \ref{percolationexponents} and \ref{sawexponents} under the assumption (\ref{perccond}), and so it will always suffice to choose
 $$ L = {C \log (1+ (1-\rho)^{-2}) \over \rho^{-1} - 1} \, ,$$
so that (\ref{perccond}) holds.

The non-backtracking random walk will be a useful tool in the proofs. Recall that this walk is simply the simple random walk not allowed to traverse back on an edge it just walked on, see the formal definition in Chapter 6 of \cite{LP}. We write $\pnbw^n(x,y)$ for the probability that the non-backtracking walk starting at $x$ visits $y$ at time $n$. Next we state two simple bounds relating $\rho$ (defined for the simple random walk) with the kernel $\pnbw$. We remark that much more precise estimates are known, but using them will only improve the possible choice of $L$ in our theorems by a multiplicative constant.

\begin{lemma} \label{nbwvssrw} For any graph $G$, vertices $x,y$ and $n \geq 0$ we have
$$ \pnbw^n(x,y) \leq \sum_{j \geq n} \psrw^j(x,y) \, .$$
\end{lemma}
\begin{proof} The non-backtracking random walk trace can be obtained from the simple random walk by sequentially erasing backtrack moves. In this coupling, if the non-backtracking walk visits $y$ at time $n$, then the simple random walk must have visited $y$ at some time which is at least $n$.
\end{proof}

\begin{lemma}\label{nbwradius} Let $G$ be an infinite graph with spectral radius $\rho < 1$. Then
$$ \pnbw^n(x,y) \leq {\rho ^n \over 1-\rho} \, .$$
\end{lemma}
\begin{proof}
It is classical that $\psrw^j(x,y) \leq \rho^j$ for all $j \geq 0$, see \cite{Woess}. This and Lemma \ref{nbwvssrw} yields the statement.
\end{proof}

\subsection{Percolation: proofs of Theorems \ref{percthm} and \ref{percolationexponents}} For an integer $n>0$ we write $\{x \stackrel{n}{\lrfill} y\}$ for the event that the shortest open path between $x$ and $y$ is of length $n$, so that $\displaystyle \prob(x \lr y) = \sum_{n=0}^{\infty} \prob(x \stackrel{n}{\lrfill} y)$. For any $p \in [0,1]$ we bound
\be\label{connbound} \prob_p(x \stackrel{n}{\lrfill} y) \leq d(d-1)^{n-1} \pnbw^n(x,y) p^n \, ,\ee
since $d(d-1)^{n-1} \pnbw^n(x,y)$ is an upper bound on the number of simple paths of length precisely $n$ between $x$ and $y$. Lemma \ref{nbwradius} implies that
$$ \prob_p(x \stackrel{n}{\lrfill} y) \leq {d [p(d-1)\rho]^n\over (d-1)(1-\rho)}  \, .$$
Hence, if $p$ is such that $p(d-1)\rho <1$ and $x,y$ are two vertices of graph distance $R$ in $G$, then
$$ \prob_p(x \lr y) = \sum_{n \geq R} \prob_p(x \stackrel{n}{\lrfill} y) \leq C [p(d-1)\rho]^R \, ,$$
where $C=C(d,\rho)>0$ is a constant. In particular $\prob_p(x \lr y)$ tends to $0$ as the graph distance in $G$ of $x$ and $y$ grows. Now, assume that (\ref{perccond}) holds. Fix $p>p_c$ so that $p(d-1)\rho<1$ and write $\theta(p) = \prob_p(|\C(0)|=\infty)>0$ for the percolation probability. By the Harris inequality we get that for any two vertices $x,y$ we have
$$ \prob_p(|\C(x)|= \infty \and |\C(y)|=\infty \and x \not \lr y) \geq \theta(p)^2 - \prob_p(x \lr y) \, .$$
We now choose $x,y$ with graph distance $R$ so large so that the last quantity is positive. Theorem 7.5 in \cite{LP} states that the number of infinite clusters is constant almost surely, and is $0,1$ or infinity. This shows that $p_c < p_u$, concluding the proof of Theorem \ref{percthm}.

We now turn to proving Theorem \ref{percolationexponents}. We use (\ref{connbound}) to bound the triangle diagram
$$ \sum_{x,y} \prob_{p_c}(0 \lr x) \prob_{p_c}(x \lr y) \prob_{p_c}(y \lr 0) $$
by
$$ \sum_{r_1,r_2,r_3=0}^\infty {d^3 \over (d-1)^3}  [p_c(d-1)]^{r_1+r_2+r_3} \sum_{x,y} \pnbw^{r_1}(0,x)\pnbw^{r_2}(x,y) \pnbw^{r_3}(y,0) \, .$$
Lemma \ref{nbwvssrw} gives
\begin{eqnarray*} \sum_{x,y} \pnbw^{r_1}(0,x)\pnbw^{r_2}(x,y) \pnbw^{r_3}(y,0) &\leq& \sum_{x,y} \sum_{\substack{n_1 \geq r_1, \\ n_2 \geq r_2, \\ n_3 \geq r_3}} \psrw^{n_1}(0,x)\psrw^{n_2}(x,y)\psrw^{n_3}(y,0) \\ &=& \sum_{\substack{n_1 \geq r_1, \\ n_2 \geq r_2, \\ n_3 \geq r_3}} \psrw^{n_1+n_2+n_3}(0,0) \leq {\rho ^{r_1+r_2+r_3} \over (1-\rho)^3} \, .\end{eqnarray*}
We get that
$$ \sum_{x,y} \prob_{p_c}(0 \lr x) \prob_{p_c}(x \lr y) \prob_{p_c}(y \lr 0) \leq C \sum_{r_1,r_2,r_3} [p_c(d-1)\rho]^{r_1+r_2+r_3} < \infty \, ,$$
concluding the proof of Theorem \ref{percolationexponents}. \qed \\

\subsection{Self-avoiding walk: proof of Theorems \ref{sawthm} and \ref{sawexponents}.}
We begin by the well known inequality that in any transitive graph $G$ we have that $\mu p_c \geq 1$. Indeed, if $p$ is such that $\mu p < 1$, then
$$ \prob_p ( 0 \lr S_n ) \leq \sum_{k\geq n} c_k p^k \longrightarrow 0 \hbox{ {\rm as} } n \to \infty \, ,$$
where $\{0 \lr S_n\}$ is the event that there exists an open path starting at $0$ and ending at the $n$-sphere $S_n = \{x : d_G(0,x)=n\}$. Since $\cap_n \{ 0 \lr S_n\}= \{ |\C(0)| = \infty\}$, we deduce that $p \leq p_c$. Hence $\mu p_c \geq 1$.

%

In the same way we derived (\ref{connbound}) we bound
\be\label{sawconn} c_n(x)  \leq d(d-1)^{n-1} \pnbw^n(x,y) \leq {d [(d-1)\rho]^n \over (d-1) (1-\rho)} \, ,\ee
where the last inequality follows from Lemma \ref{nbwradius}. Thus for any small $\eps>0$ there exists $n_0$ such that for all $n \geq n_0$
$$ {c_n(x) \over c_n} \leq C [ (\mu^{-1} + \eps) (d-1) \rho]^n \, ,$$
where $C=C(\rho)>0$ is a constant. Assume now that (\ref{perccond}) holds, since $\mu p_c \geq 1$ we deduce that we can choose $\eps>0$ small enough so that the base of the exponent in the previous inequality is less than $1$. This concludes the first assertion of Theorem \ref{sawthm} that $\psaw^n(0,x)$ decays exponentially in $n$ uniformly in $x$. This exponential decays also establishes positive speed for the self-avoiding walk since for any $\alpha>0$ we have
$$ \sum_{x : d_G(0,x) \leq \alpha n} {c_n(x) \over c_n} \leq C(d-1)^{\alpha n} \sup_{x\in G} {c_n(x) \over c_n} \, .$$
Now choose $\alpha=\alpha(\rho,d)>0$ small enough so that the right hand side converges to $0$. This finishes the proof of Theorem \ref{sawthm}.

We now turn to prove Theorem \ref{sawexponents}. We write
$$ \sum_{x \in G} G_z^2(x) = \sum_{x \in G} \sum_{n \geq 0, m\geq 0} c_n(x)c_m(x) z^{n+m} \, ,$$
which is valid as long as the sums converge. We use the first estimate in (\ref{sawconn}) to bound
$$ \sum_{x \in G} G_z^2(x) \leq C \sum_{n \geq 0, m\geq 0} [z(d-1)]^{n+m} \sum_{x} \pnbw^n(0,x) \pnbw^m(0,x) \, .$$
By Lemma \ref{nbwvssrw} we obtain the bound
\begin{eqnarray*}
\sum_{x \in G} G_z^2(x) &\leq& C \sum_{n \geq 0, m\geq 0} [z(d-1)]^{n+m} \sum_{n_1 \geq n, m_1 \geq m} \sum_{x} \psrw^{n_1}(0,x) \psrw^{m_1}(0,x) \\ &\leq& C \sum_{n \geq 0, m\geq 0} [z(d-1)]^{n+m} \sum_{n_1 \geq n, m_1 \geq m} \psrw^{n_1+m_1}(0,0) \\ &\leq& C(1-\rho)^{-2} \sum_{n \geq 0, m\geq 0} [z(d-1) \rho]^{n+m} \, . \end{eqnarray*}
Now, when (\ref{perccond}) holds, since $\mu p_c \geq 1$ we find that
$$ \sum_{x\in G} G_{\mu^{-1}}^2(x) < \infty \, ,$$
concluding the proof of Theorem \ref{sawexponents}. \qed\\

%
%
%

\vspace{.05 in}\noindent
{\bf Asaf Nachmias}: \texttt{asafn(at)math.ubc.ca} \\
Department of Mathematics, University of British Columbia, 121-1984 Mathematics Rd,
Vancouver, BC, Canada V6T1Z2.

\medskip \noindent
{\bf Yuval Peres}: \texttt{peres(at)microsoft.com} \\
Microsoft Research, One Microsoft way, Redmond, WA 98052-6399, USA.

\end{document}